\documentclass[a4paper,11pt]{article}

\usepackage[english]{babel}
\usepackage[utf8]{inputenc}
\usepackage[T1]{fontenc}
\usepackage{amsmath,amsthm,amsfonts,amssymb}
\usepackage{geometry}
\usepackage{enumitem}

\newtheorem{theorem}{Theorem}[section]
\newtheorem{definition}[theorem]{Definition}
\newtheorem{lemma}[theorem]{Lemma}

\newcommand{\norm}[1]{\left\lVert#1\right\rVert}
\newcommand{\inner}[2]{\left\langle#1, #2\right\rangle}
\newcommand{\Aff}{\operatorname{Aff}}

\title{A Probabilistic Generalization of the Mazur-Ulam Theorem}
\author{Justinas Zaliaduonis \and Sergios Gatidis}
\date{\today}

\begin{document}

\maketitle

\begin{abstract}
The classical Mazur-Ulam theorem establishes that every surjective isometry between normed real vector spaces is an affine transformation. In various applied mathematical settings, however, one encounters maps that preserve distances not pointwise, but \emph{almost everywhere} with respect to a probability measure. This paper provides a rigorous generalization of the Mazur-Ulam theorem to probability spaces. We prove that if a measurable map on a subset of $\mathbb{R}^d$ preserves distances almost everywhere with respect to a measure with full-dimensional support, it coincides almost everywhere with a global Euclidean isometry, defined as an orthogonal transformation followed by a translation.
\end{abstract}

\section{Introduction}

The study of isometries is central to the geometry of Banach spaces. The foundational result in this area is the Mazur-Ulam theorem \cite{mazur1932}, which characterizes the rigidity of distance-preserving maps.

\begin{theorem}[Mazur-Ulam]
    \label{thm:mazur-ulam-classical}
    Let $V$ and $W$ be normed vector spaces over $\mathbb{R}$. If $f: V \to W$ is a surjective isometry, then $f$ is an affine map. That is, $f(x) = T(x) + b$ where $T$ is a linear map and $b \in W$ is a translation vector.
\end{theorem}

In the context of Euclidean spaces $\mathbb{R}^d$, this implies that any surjective distance-preserving map is a composition of an orthogonal linear transformation and a translation. \\

However, strict pointwise constraints are often too restrictive for measure-theoretic contexts. A natural question arises: if a map behaves like an isometry "for the most part", does it retain the rigid structure of a true isometry? \\

In this work, we address the relaxation of the pointwise isometry condition to an \emph{almost everywhere} condition with respect to a probability measure $\mu$. We seek to answer the following:

\begin{quote}
    \emph{Let $\mu$ be a probability measure on $\mathbb{R}^d$ with full-dimensional support. If a measurable map $h: \mathbb{R}^d \to \mathbb{R}^d$ preserves distances for $\mu \times \mu$-almost all pairs of points, does there exist a global affine isometry $H$ such that $h = H$ almost everywhere?}
\end{quote}

We answer this question in the affirmative. Our approach bridges affine geometry and measure theory. We first define the concept of an isometry almost everywhere. We then utilize the geometric properties of simplexes in $\mathbb{R}^d$ to construct a candidate global isometry and use Fubini's theorem to extend the equality to the full measure space.

\section{Preliminaries and Definitions}

We assume the ambient space is the Euclidean space $\mathbb{R}^d$ equipped with the standard inner product $\inner{\cdot}{\cdot}$ and norm $\norm{\cdot}$. Let $\mathcal{B}(\mathbb{R}^d)$ denote the Borel $\sigma$-algebra.

\subsection{Measure Theoretic Setup}

\begin{definition}[Full-Dimensional Support]
    Let $\mu$ be a probability measure on $(\mathbb{R}^d, \mathcal{B}(\mathbb{R}^d))$. The support of $\mu$, denoted $\operatorname{supp}(\mu)$, is the set of all points $x \in \mathbb{R}^d$ for which every open neighborhood of $x$ has positive measure. We say $\mu$ has \emph{full-dimensional support} if $\operatorname{supp}(\mu)$ is not contained in any proper affine hyperplane of $\mathbb{R}^d$.
\end{definition}

\begin{definition}[Isometry Almost Everywhere]
\label{def:isometry-almost-everywhere}
    Let $(Z, \mathcal{A}, \mu)$ be a probability space where $Z \subseteq \mathbb{R}^d$. A measurable map $h: Z \to \mathbb{R}^d$ is called an \emph{isometry almost everywhere} (or $\mu$-a.e. isometry) if there exists a set $D \in \mathcal{A} \otimes \mathcal{A}$ with $(\mu \times \mu)(D^c) = 0$ such that:
    \begin{equation}
        \norm{h(z) - h(\tilde{z})} = \norm{z - \tilde{z}} \quad \text{for all } (z, \tilde{z}) \in D.
    \end{equation}
\end{definition}

\subsection{Affine Geometry}

To prove the main result, we rely on the rigidity of finite geometric structures.

\begin{definition}[Affine Independence]
    A set of $k+1$ points $\{v_0, v_1, \dots, v_k\} \subset \mathbb{R}^d$ is \emph{affinely independent} if the vectors $\{v_1 - v_0, \dots, v_k - v_0\}$ are linearly independent.
\end{definition}

\begin{definition}[Affine Hull]
    For a subset $S \subseteq \mathbb{R}^d$, the affine hull is defined as:
    \[
    \Aff(S) = \left\{ \sum_{i=1}^k \alpha_i x_i \mid k \in \mathbb{N}, x_i \in S, \sum_{i=1}^k \alpha_i = 1 \right\}.
    \]
\end{definition}

A set $S$ has full dimension in $\mathbb{R}^d$ if $\Aff(S) = \mathbb{R}^d$. This requires $S$ to contain at least $d+1$ affinely independent points.

\section{Rigidity of Finite Isometries}

Before addressing the probabilistic case, we establish that an isometry defined on a sufficiently rich finite set uniquely determines a global affine isometry. This is a crucial step in "locking" the geometry of the map.

\begin{lemma}[Extension of Finite Isometry]
\label{lemma:finite_extension}
    Let $S = \{a_0, a_1, \dots, a_d\} \subset \mathbb{R}^d$ be a set of $d+1$ affinely independent points. Let $f: S \to \mathbb{R}^d$ be a map that preserves pairwise distances on $S$, i.e., $\norm{f(a_i) - f(a_j)} = \norm{a_i - a_j}$ for all $0 \leq i, j \leq d$.
    
    Then, there exists a unique global Euclidean isometry $H: \mathbb{R}^d \to \mathbb{R}^d$ of the form $H(x) = Qx + b$, with $Q \in O(d)$ (the orthogonal group) and $b \in \mathbb{R}^d$, such that $H|_{S} = f$.
\end{lemma}

\begin{proof}
    Define the translation vectors $v_i = a_i - a_0$ for $i = 1, \dots, d$. Since $S$ is affinely independent, $\{v_1, \dots, v_d\}$ forms a basis for $\mathbb{R}^d$.
    
    Similarly, define $w_i = f(a_i) - f(a_0)$. By the distance preserving property:
    \[ \norm{w_i} = \norm{f(a_i) - f(a_0)} = \norm{a_i - a_0} = \norm{v_i}. \]
    Furthermore, for any $i, j$:
    \[ \norm{w_i - w_j} = \norm{f(a_i) - f(a_j)} = \norm{a_i - a_j} = \norm{v_i - v_j}. \]
    
    Using the polarization identity $\inner{x}{y} = \frac{1}{2}(\norm{x}^2 + \norm{y}^2 - \norm{x-y}^2)$, we deduce that inner products are preserved:
    \begin{align*}
        \inner{w_i}{w_j} &= \frac{1}{2} (\norm{w_i}^2 + \norm{w_j}^2 - \norm{w_i - w_j}^2) \\
        &= \frac{1}{2} (\norm{v_i}^2 + \norm{v_j}^2 - \norm{v_i - v_j}^2) \\
        &= \inner{v_i}{v_j}.
    \end{align*}
    
    Let $V$ be the matrix with columns $v_i$ and $W$ be the matrix with columns $w_i$. The condition $\inner{v_i}{v_j} = \inner{w_i}{w_j}$ implies that the Gram matrices are identical: $V^T V = W^T W$. Since $v_i$ form a basis, $V$ is invertible, and we can define a linear map $Q$ such that $Q v_i = w_i$. Specifically, $Q = W V^{-1}$.
    
    To show $Q$ is orthogonal, consider any $x \in \mathbb{R}^d$. We can write $x = \sum c_i v_i$. Then:
    \[ \norm{Qx}^2 = \norm{\sum c_i w_i}^2 = \sum_{i,j} c_i c_j \inner{w_i}{w_j} = \sum_{i,j} c_i c_j \inner{v_i}{v_j} = \norm{\sum c_i v_i}^2 = \norm{x}^2. \]
    Since $Q$ is linear and preserves norms, $Q \in O(d)$.
    
    Finally, define the affine map $H(x) = Q(x - a_0) + f(a_0)$. By construction:
    \[ H(a_0) = Q(0) + f(a_0) = f(a_0), \]
    \[ H(a_i) = Q(v_i) + f(a_0) = w_i + f(a_0) = f(a_i) - f(a_0) + f(a_0) = f(a_i). \]
    
    Uniqueness follows from the fact that an affine map is uniquely determined by its action on a $d$-simplex (an affinely independent set of $d+1$ points).
\end{proof}

\section{Main Result: Probabilistic Mazur-Ulam Theorem}

We now state and prove the generalization of the Mazur-Ulam theorem for probability spaces.

\begin{theorem}[Probabilistic Mazur-Ulam]
    \label{thm:probabilistic_mazur_ulam}
    Let $\mu$ be a probability measure on $\mathbb{R}^d$ with full-dimensional support. Let $h: \mathbb{R}^d \to \mathbb{R}^d$ be a measurable function that is an isometry almost everywhere with respect to $\mu$. \\
    
    Then, there exists a global Euclidean isometry $H: \mathbb{R}^d \to \mathbb{R}^d$ (where $H(x) = Ax + b$ with $A^T A = I$ and $b \in \mathbb{R}^d$) such that:
    \[ h(x) = H(x) \quad \text{for } \mu\text{-almost all } x. \]
\end{theorem}

\begin{proof}
    Let $D \subset \mathbb{R}^d \times \mathbb{R}^d$ be the set of full measure such that $(\mu \times \mu)(D^c)=0$ and the isometry property holds on $D$:
    \[ \norm{h(x) - h(y)} = \norm{x - y} \quad \forall (x,y) \in D. \] 
    
    \textbf{Step 1: Selecting a Basis.}
    We define the sections of $D$. For any $x \in \mathbb{R}^d$, let $D_x = \{ y \in \mathbb{R}^d \mid (x,y) \in D \}$. By Fubini's theorem, since $(\mu \times \mu)(D^c) = 0$, there exists a set $X_0 \subset \mathbb{R}^d$ with $\mu(X_0) = 1$ such that for all $x \in X_0$, $\mu(D_x) = 1$. \\
    
    Since $\mu$ has full-dimensional support, $X_0$ is not contained in any proper affine hyperplane. We can iteratively select a set of $d+1$ points $\mathcal{S} = \{a_0, a_1, \dots, a_d\}$ satisfying two conditions: (1) $\mathcal{S}$ is affinely independent, and (2) all pairs in $\mathcal{S}$ belong to $D$. \\
    
    We construct $\mathcal{S}$ inductively. Select $a_0 \in X_0$. Let $A_0 = D_{a_0} \cap X_0$. This set has measure 1.
    For $k < d$, assume $a_0, \dots, a_k$ have been selected. Define $A_k = (\cap_{i=0}^k D_{a_i}) \cap X_0$. Since this is a finite intersection of measure 1 sets, $\mu(A_k) = 1$. Because $\mu$ has full affine support, $A_k$ is not contained in the affine span of $\{a_0, \dots, a_k\}$. Choose $a_{k+1} \in A_k$ such that it is affinely independent of the previous points. \\
    
    The resulting set $\mathcal{S} = \{a_0, \dots, a_d\}$ is affinely independent, and for all $a_i, a_j \in \mathcal{S}$, $(a_i, a_j) \in D$, meaning $\norm{h(a_i) - h(a_j)} = \norm{a_i - a_j}$. \\
    
    \textbf{Step 2: Constructing the Global Isometry.}
    By Lemma \ref{lemma:finite_extension}, the restriction of $h$ to $\mathcal{S}$ extends uniquely to a global Euclidean isometry $H: \mathbb{R}^d \to \mathbb{R}^d$, defined by $H(x) = Qx + b$. By construction, $H(a_i) = h(a_i)$ for all $i=0, \dots, d$. \\
    
    \textbf{Step 3: Verification Almost Everywhere.}
    We claim that $h(z) = H(z)$ for $\mu$-almost all $z$.
    Consider the set:
    \[ G = \bigcap_{i=0}^d D_{a_i}. \]
    Since $\mu(D_{a_i}) = 1$ for each $i$, the intersection $G$ has $\mu(G) = 1$.
    
    For any $z \in G$, we have by definition that $(z, a_i) \in D$ for all $i=0, \dots, d$. Therefore:
    \[ \norm{h(z) - h(a_i)} = \norm{z - a_i} \quad \forall i \in \{0, \dots, d\}. \]
    Substituting $h(a_i) = H(a_i)$ (from Step 2):
    \[ \norm{h(z) - H(a_i)} = \norm{z - a_i}. \]
    
    Since $H$ is a global isometry, it preserves distances everywhere. Thus:
    \[ \norm{H(z) - H(a_i)} = \norm{z - a_i}. \]
    
    Combining these equalities, we find that for any $z \in G$:
    \[ \norm{h(z) - H(a_i)} = \norm{H(z) - H(a_i)} \quad \forall i \in \{0, \dots, d\}. \]
    
    This implies that the point $h(z)$ and the point $H(z)$ are equidistant from the $d+1$ points $\{H(a_0), \dots, H(a_d)\}$. Since $H$ is an affine isomorphism, the set $\{H(a_0), \dots, H(a_d)\}$ is affinely independent.
    
    In Euclidean space $\mathbb{R}^d$, a point is uniquely determined by its distances to $d+1$ affinely independent points. Therefore, for all $z \in G$, we must have $h(z) = H(z)$. Since $\mu(G) = 1$, the proof is complete.
\end{proof}

\section{Conclusion}

We have presented a rigorous extension of the Mazur-Ulam theorem to the setting of probability spaces. By leveraging the geometric rigidity of simplexes and the properties of measure intersections, we demonstrated that measure-theoretic distance preservation implies global geometric rigidity given full-dimensional support. This result bridges the gap between probabilistic functional analysis and classical affine geometry, ensuring that maps which appear isometric to a random observer are, in fact, structurally isometric almost everywhere.

\end{document}